\documentclass[12pt]{amsart}

\usepackage{accents}
\usepackage{appendix}
\usepackage{amsfonts}
\usepackage{amsmath}
\usepackage{amssymb}	
\usepackage{amsthm,bm}
\usepackage{array,booktabs,multirow}
\usepackage{braket}
\usepackage{centernot}
\usepackage{cite}
\usepackage{comment}
\usepackage{dsfont}
\usepackage{mathalfa}
\usepackage[shortlabels]{enumitem}
\usepackage{etoolbox}
\usepackage{float}
\usepackage[hang, flushmargin]{footmisc}
\usepackage{latexsym}
\usepackage{lipsum}
\usepackage{needspace}
\usepackage{tikz}
\usepackage{hyperref}
\usetikzlibrary{matrix,arrows}

\makeatletter
\def\namedlabel#1#2{\begingroup
   \def\@currentlabel{#2}%
   \label{#1}\endgroup
}
\makeatother

\theoremstyle{plain}
\newtheorem{thm}{Theorem}[section]

\newtheorem{lem}[thm]{Lemma}
\newtheorem{prop}[thm]{Proposition}

\theoremstyle{definition}

\newtheorem{defn}[thm]{Definition}

\theoremstyle{remark}

\setlist[enumerate,1]{leftmargin=2em}

\def\C{\mathbb C}
\def\End{{\rm End}}

\def\Mat{{\rm Mat}}
\def\M{\mathcal M}

\def\T{\mathcal T}
\def\F{\mathbb F}

\newcommand{\floor}[1]{\left\lfloor #1 \right\rfloor}
\newcommand{\ceil}[1]{\left\lceil #1 \right\rceil}

\title[The Terwilliger algebra of the halved cube]{The Terwilliger algebra of the halved cube}

\author{Chia-Yi Wen}
\address{
Department of Applied Mathematics\\
National Yang Ming Chiao Tung University\\
Hsinchu 30010 Taiwan
}
\email{cywen.am06g@nctu.edu.tw}

\author{Hau-Wen Huang}
\address{
Department of Mathematics\\
National Central University\\
Chung-Li 32001 Taiwan
}
\email{hauwenh@math.ncu.edu.tw}

\begin{document}
\begin{abstract}
Let $D\geq 3$ denote an integer. For any $x\in \F_2^D$ let $w(x)$ denote the Hamming weight of $x$. Let  $X$ denote the subspace of $\F_2^D$ consisting of all $x\in \F_2^D$ with even $w(x)$. The $D$-dimensional halved cube $\frac{1}{2}H(D,2)$ is a finite simple connected graph with vertex set $X$ and $x,y\in X$ are adjacent if and only if $w(x-y)=2$.
Fix a vertex $x\in X$.
The Terwilliger algebra $\mathcal T=\mathcal T(x)$ of $\frac{1}{2}H(D,2)$ with respect to $x$ is the subalgebra of $\Mat_X(\C)$ generated by the adjacency matrix $A$ and the dual adjacency matrix $A^*=A^*(x)$ where $A^*$ is a diagonal matrix with
$$
A^*_{yy}=D-2w(x-y)
\qquad
\hbox{for all $y\in X$}.
$$
In this paper we decompose the standard $\T$-module into a direct sum of irreducible $\T$-modules.
\end{abstract}

\maketitle



\allowdisplaybreaks

\section{Preliminaries}\label{s:preliminary}
Throughout this paper, we adopt the following conventions: An algebra is meant to be a unital associative algebra. A subalgebra has the same unit as the parent algebra.
Let $\C$ denote the complex number field.
Given a finite set $Y\not=\emptyset$, let $\Mat_Y(\C)$ denote the algebra consisting of the square matrices over $\C$ indexed by $Y$ and let $\C^Y$ denote the vector space consisting of all column vectors over $\C$ indexed by $Y$.

Let $\Gamma$ denote a finite simple connected graph with vertex set $X\not=\emptyset$. Let $\partial$ denote the path-length distance function of $\Gamma$. The {\it diameter} $D$ of $\Gamma$ is defined by
$$
D=\max_{x,y\in X}\partial(x,y).
$$
Given any $x\in X$ let
$$
\Gamma_i(x)=\{y\in X\,|\, \partial(x,y)=i\}
\qquad
\hbox{for $i=0,1,\ldots,D$}.
$$
For short we abbreviate $\Gamma(x)=\Gamma_1(x)$. We call $\Gamma$ {\it distance-regular} whenever for all $h,i,j\in\{0,1,\ldots,D\}$ and all $x,y\in X$ with $\partial(x,y)=h$ the number
$
|\Gamma_i(x)\cap \Gamma_j(y)|
$
is independent of $x$ and $y$. If $\Gamma$ is distance-regular, then the numbers $a_i,b_i,c_i$ for all $i=0,1,\ldots,D$ defined by
$$
a_i = |\Gamma_i(x) \cap \Gamma(y)|,
\qquad
b_i = |\Gamma_{i+1}(x) \cap \Gamma(y)|,
\qquad
c_i = |\Gamma_{i-1}(x) \cap \Gamma(y)|
$$
for any $x,y\in X$ with $\partial(x,y)=i$ are called the {\it intersection numbers} of $\Gamma$. Here $\Gamma_{-1}(x)$ and $\Gamma_{D+1}(x)$ are empty sets.

Now assume that $\Gamma$ is distance-regular. For all $i=0,1,\ldots,D$ the $i^{\,th}$ {\it distance matrix} $A_i$ of $\Gamma$ is a $0$-$1$ matrix in $\Mat_{X}(\C)$ defined by
\begin{equation*}
(A_i)_{xy}=\left\{
  \begin{array}{ll}
   1  & \hbox{if $\partial(x,y)=i$,} \\
   0  & \hbox{if $\partial(x,y)\neq i$}
  \end{array}
\right.
\qquad
\hbox{for all $x,y\in X$}.
\end{equation*}
We abbreviate $A=A_1$ and $A$ is called the {\it adjacency matrix} of $\Gamma$. Let $a_i,b_i,c_i$ ($0\leq i\leq D$) denote the intersection numbers of $\Gamma$. Observe that
\begin{gather*}
A_i A=b_{i-1} A_{i-1}+a_i A_i+c_{i+1} A_{i+1}
\qquad
\hbox{for all $i=0,1,\ldots, D$},
\end{gather*}
where $b_{-1} A_{-1}$ and $c_{D+1} A_{D+1}$ are interpreted as the zero matrix in $\Mat_X(\C)$.
The {\it Bose--Mesner algebra} $\mathcal{M}$ of $\Gamma$ is the subalgebra of $\Mat_X(\C)$ generated by $A$. Note that  the matrices $\{A_i\}_{i=0}^{D}$ form a basis for $\M$. Let $\circ$ denote the entrywise product. Since
$$
A_i\circ A_j=\delta_{ij} A_i
\qquad (0\leq i,j\leq D),
$$
it follows that $\M$ is closed under $\circ$.

Since $A$ is a real symmetric matrix and $\M$ has dimension $D+1$, it follows that $A$ has $D+1$ distinct real eigenvalues $\theta_0,\theta_1,\ldots,\theta_D$. There are unique $E_0,E_1,\ldots,E_D\in \Mat_X(\C)$ satisfying
\begin{align*}
\sum_{i=0}^D E_i=I,
\qquad
A E_i=\theta_i E_i
\qquad
\hbox{for all $i=0,1,\ldots, D$},
\end{align*}
where $I$ denotes the identity matrix in $\Mat_X(\C)$.
Note that $\{E_i\}_{i=0}^D$ form a basis for $\M$.
The element $E_i$ is called the {\it primitive idempotent} of $\Gamma$ associated with $\theta_i$ for $i=0,1,\ldots,D$.

The distance-regular graph $\Gamma$ is said to be {\it $Q$-polynomial} with respect to the ordering $\{E_i\}_{i=0}^D$ if there are unique $a_i^*,b_i^*,c_i^*\in\C$ for all $i=0,1,\ldots,D$ with
$b_D^*=c_0^*=0$, $b_{i-1}^*c_i^*\neq 0$ for all $i=1,2,\ldots,D$
such that
$$
E_i\circ E_1=\frac{1}{|X|}\left(b_{i-1}^*E_{i-1}+a_i^*E_i+c_{i+1}^*E_{i+1}\right)\qquad
\hbox{for all $i=0,1,\ldots,D$},
$$
where $b_{-1}^*E_{-1}$ and $c_{D+1}^*E_{D+1}$ are interpreted as the zero matrix in $\Mat_X(\C)$. If this is the case, then $a_i^*,b_i^*,c_i^*$ for all $i=0,1,\ldots,D$ are called the {\it dual intersection numbers} of $\Gamma$.

Now assume that $\Gamma$ is a $Q$-polynomial distance-regular graph.
Fix a vertex $x\in X$. For $i=0,1,\ldots, D$ the $i^{\, th}$ {\it dual distance matrix} $A_i^*=A_i^*(x)$ of $\Gamma$ respect to $x$ is the diagonal matrix in $\Mat_X(\C)$ defined by
\begin{equation*}
(A_i^*)_{yy}=|X|(E_i)_{xy}\qquad\hbox{for all $y\in X$}.
\end{equation*}
We abbreviate $A^*=A_1^*$ and $A^*$ is called the {\it dual adjacency matrix} of $\Gamma$ with respect to $x$. Let $a_i^*,b_i^*,c_i^*$ $(0\leq i\leq D)$ denote the dual intersection numbers of $\Gamma$. Observe that
\begin{gather*}
A_i^* A^*=b_{i-1}^*A_{i-1}^*+a_i^*A_i^*+c_{i+1}^*A_{i+1}^*
\qquad
\hbox{for all $i=0,1,\ldots,D$},
\end{gather*}
where $b_{-1}^* A_{-1}^*$ and $c_{D+1}^* A_{D+1}^*$ are interpreted as the zero matrix in $\Mat_X(\C)$. The {\it dual Bose--Mesner algebra} $\M^*=\M^*(x)$ of $\Gamma$ with respect to $x$ is the subalgebra of $\Mat_X(\C)$ generated by $A^*$.  Note that $\{A_i^*\}_{i=0}^{D}$ form a basis for $\M^*$.

The {\it Terwilliger algebra} $\T=\T(x)$ of $\Gamma$ with respect to $x$ is the subalgebra of $\Mat_X(\C)$ generated by $\M$ and $\M^*$ \cite{TerAlgebraI,TerAlgebraII,TerAlgebraIII}. The vector space $\C^X$ has a natural $\T$-module structure and it is called the {\it standard} $\T$-module.
Since $\T$ is finite-dimensional the irreducible $\T$-modules are finite-dimensional. Since $\T$ is closed under the conjugate-transpose map the algebra $\T$ is semisimple. Hence the algebra $\T$ is isomorphic to
$$
\bigoplus_{\tiny \hbox{irreducible $\T$-modules $V$}} \End(V),
$$
where the direct sum is over all  non-isomorphic irreducible $\T$-modules $V$. Since the standard $\T$-module $\C^X$ is faithful it follows that $\C^X$ contains all irreducible $\T$-modules up to isomorphism.

The paper is organized as follows:
In \S\ref{s:introduction} we state the main results of this paper on the Terwilliger algebra of the halved cube (Proposition \ref{prop:irreducible Mk} and Theorems \ref{thm:decomposition of halved cube}--\ref{thm:structure of T}).
In \S\ref{s:hypercube} we recall some results concerning the Terwilliger algebra of the hypercube. In \S\ref{sec:standard T module of halved cube} we relate the Terwilliger algebra of the halved cube to a subalgebra of  the Terwilliger algebra of the hypercube. In \S\ref{sec:decomposition} we give the proofs of our main results.

\section{Statement of results}\label{s:introduction}

Let $D\geq 3$ denote an integer.
Let $\F_2$ denote the finite field of order two. Let $\F_2^D$ denote a $D$-ary Cartesian product of $\F_2$.
For any $x\in \F_2^D$ the {\it Hamming weight} $w(x)$ of $x$ is the number of ones in $x$. Let $X$ denote the subspace of $\F_2^D$ consisting of all $x\in \F_2^D$ with even $w(x)$.

\begin{defn}\label{defn:halved cube}
The {\it $D$-dimensional halved cube} $\frac{1}{2}H(D,2)$ is a finite simple connected graph with vertex set $X$ and $x,y\in X$ are adjacent if and only if $w(x-y)=2$.
\end{defn}

Note that $\frac{1}{2}H(D,2)$ is a distance-regular graph of diameter $\floor{\frac{D}{2}}$ whose intersection numbers \cite{BannaiIto1984,DRG_book1989,TerAlgebraIII} are
\begin{equation*}
a_i=2i(D-2i),\qquad
b_i=\binom{D-2i}{2},\qquad
c_i=\binom{2i}{2}
\qquad \hbox{for $i=0,1,\ldots,\floor{\frac{D}{2}}$}.
\end{equation*}
Let $A$ denote the adjacency matrix of $\frac{1}{2}H(D,2)$. The eigenvalues of $A$ are
$$
\theta_i=\frac{1}{2}((D-2i)^2-D)
\qquad
\hbox{for $i=0,1,\ldots,\floor{\frac{D}{2}}$}.
$$
Let $E_i$ denote the primitive idempotent of $\frac{1}{2}H(D,2)$ associated with $\theta_i$ for $i=0,1,\ldots,\floor{\frac{D}{2}}$.

Note that $\frac{1}{2}H(D,2)$ is $Q$-polynomial with respect to the ordering $\{E_i\}_{i=0}^{\floor{\frac{D}{2}}}$ whose dual intersection numbers \cite{BannaiIto1984,TerAlgebraIII} are
\begin{align*}
&a_i^*=0,\qquad b_i^*=D-i,\qquad c_i^*=i
\qquad
\hbox{for $i=0,1,\ldots,\floor{\frac{D}{2}}-1$},
\\
&a_{\floor{\frac{D}{2}}}^*=
\left\{
\begin{array}{ll}
0 \qquad &\hbox{if $D$ is even},
\\
\frac{D+1}{2}
 \qquad &\hbox{if $D$ is odd},
\end{array}
\right.
\qquad 
 b_{\floor{\frac{D}{2}}}^*=0,
\qquad 
c_{\floor{\frac{D}{2}}}^*=
\left\{
\begin{array}{ll}
D \qquad &\hbox{if $D$ is even},
\\
\frac{D-1}{2}
 \qquad &\hbox{if $D$ is odd}.
\end{array}
\right.
\end{align*}

Fix a vertex $x\in X$. Let $A^*=A^*(x)$ denote the dual adjacency matrix of $\frac{1}{2}H(D,2)$ with respect to $x$. The diagonal matrix $A^*$ is given by
$$
A^*_{yy}=D-2w(x-y)
\qquad
\hbox{for all $y\in X$}.
$$
Let $\T=\T(x)$ denote the Terwilliger algebra of $\frac{1}{2}H(D,2)$ with respect to $x$. The main results of this paper are as follows:

\begin{prop}\label{prop:irreducible Mk}
\begin{enumerate}
\item For any even integer $k$ with $0\leq k\leq \floor{\frac{D}{2}}$, there exists a $\left(\floor{\frac{D}{2}}-k+1\right)$-dimensional irreducible $\T$-module $M_k$ that has a basis with respect to which the matrix representing $A$ and $A^*$ are
\begin{gather}\label{eq:irr module even}
\begin{pmatrix}
\alpha_0 &\gamma_1 & &  &{\bf 0}
\\
\beta_0 &\alpha_1 &\gamma_2
\\
&\beta_1 &\alpha_2 &\ddots
 \\
& &\ddots &\ddots &\gamma_{\floor{\frac{D}{2}}-k}
 \\
{\bf 0} & & &\beta_{\floor{\frac{D}{2}}-k-1} &\alpha_{\floor{\frac{D}{2}}-k}
\end{pmatrix},
\qquad
\begin{pmatrix}
\theta_0^* & &  & &{\bf 0}
\\
 &\theta_1^*
\\
 &  &\theta_2^*
 \\
 & & &\ddots
 \\
{\bf 0}  & & & &\theta_{\floor{\frac{D}{2}}-k}^*
\end{pmatrix},
\end{gather}
respectively, where
\begin{align*}
  \alpha_i=&2i(D-2i-2k)-k
   \qquad
   \left(0\leq i\leq \textstyle \floor{\frac{D}{2}}-k\right),
   \\
  \beta_i =& (i+1)(2i+1)
  \qquad
  \left(0\leq i\leq \textstyle \floor{\frac{D}{2}}-k-1\right),
  \\
  \gamma_i =& \frac{1}{2}(D-2i-2k+1) (D-2i-2k+2)
  \qquad
  \left(0\leq i\leq \textstyle \floor{\frac{D}{2}}-k\right),
  \\
  \theta_i^* =& D-2(2i+k)
  \qquad
  \left(0\leq i\leq \textstyle \floor{\frac{D}{2}}-k\right).
\end{align*}
  \item
For any even integer $k$ with $2\leq k\leq \ceil{\frac{D}{2}}$, there exists a $\left(\ceil{\frac{D}{2}}-k+1\right)$-dimensional irreducible $\T$-module $N_k$ that has a basis with respect to which the matrix representing $A$ and $A^*$ are
      \begin{equation}\label{eq:irr module odd}
      \begin{pmatrix}
       \alpha_0 & \gamma_1 &          &        & \bm{0} \\
       \beta_0  & \alpha_1 & \gamma_2 &        &        \\
                & \beta_1  & \alpha_2 & \ddots &        \\
                &          & \ddots   & \ddots & \gamma_{\ceil{\frac{D}{2}}-k} \\
         \bm{0} &          &          & \beta_{\ceil{\frac{D}{2}}-k-1} & \alpha_{\ceil{\frac{D}{2}}-k}
     \end{pmatrix},
     \qquad
    \begin{pmatrix}
       \theta_0^* &  &  &    &  \\
        & \theta_1^* &  &    &  \\
        &  & \theta_2^* &    &  \\
        &  &  & \ddots   &  \\
        &  &  &    & \theta_{\ceil{\frac{D}{2}}-k}^*
     \end{pmatrix},\end{equation}
     respectively, where
\begin{align*}
  \alpha_i=&(2i+1)(D-2i-2k+1)-k+1 \qquad
   \left(0\leq i\leq \textstyle \ceil{\frac{D}{2}}-k\right),\\
  \beta_i =& (i+1)(2i+3) \qquad
   \left(0\leq i\leq \textstyle \ceil{\frac{D}{2}}-k-1\right), \\
  \gamma_i =& \frac{1}{2}(D-2i-2k+2)(D-2i-2k+3) \qquad
   \left(0\leq i\leq \textstyle \ceil{\frac{D}{2}}-k\right), \\
  \theta_i^* =& D-2(2i+k) \qquad
   \left(0\leq i\leq \textstyle \ceil{\frac{D}{2}}-k\right).
\end{align*}
\end{enumerate}
\end{prop}

Given a vector space $V$ and a positive integer $p$, we let
$$
p\cdot V=\underbrace{V\oplus V\oplus \cdots \oplus V}_
{\hbox{{\tiny $p$ copies of $V$}}}.
$$

\begin{thm}\label{thm:decomposition of halved cube}
The standard $\T$-module $\C^{X}$ is isomorphic to
$$
\bigoplus_{\substack{k=0 \\ \hbox{\tiny $k$ is even}}}^{\floor{\frac{D}{2}}}
\left(
\frac{D-2k+1}{D-k+1}{D\choose k}\cdot M_k
\right)
\oplus
\bigoplus_{\substack{k=2 \\ \hbox{\tiny $k$ is even}}}^{\ceil{\frac{D}{2}}}
\left(
\frac{D-2k+3}{D-k+2}{D\choose k-1}\cdot N_k
\right).
$$
\end{thm}

\begin{thm}\label{thm:classification of irreducible modules}
The $\T$-modules
\begin{align*}
M_k \qquad
&\hbox{for all even integers $k$ with $0\leq k\leq \floor{\frac{D}{2}}$},
\\
N_k \qquad
&\hbox{for all even integers $k$ with $2\leq k\leq \ceil{\frac{D}{2}}$}
\end{align*}
are all irreducible $\T$-modules up to isomorphism. Moreover these $\T$-modules are mutually non-isomorphic.
\end{thm}

\begin{thm}\label{thm:structure of T}
The algebra $\T$ is isomorphic to
$$
\bigoplus_{\substack{k=0 \\ \hbox{\tiny $k$ is even}}}^{\floor{\frac{D}{2}}}
\Mat_{\floor{\frac{D}{2}}-k+1}(\C)
\oplus
\bigoplus_{\substack{k=2 \\ \hbox{\tiny $k$ is even}}}^{\ceil{\frac{D}{2}}}
\Mat_{\ceil{\frac{D}{2}}-k+1}(\C).
$$
In particular
$
\dim \T={\floor{\frac{D}{2}}+3 \choose 3}+{\ceil{\frac{D}{2}}+1 \choose 3}$.
\end{thm}

Note that the notations of this section will be used in the rest of this paper.

\section{The Terwilliger algebra of the hypercube}\label{s:hypercube}

\begin{defn}\label{defn:cube}
The {\it $D$-dimensional hypercube} $H(D,2)$ is a finite simple connected graph with vertex set $\F_2^D$ and $x,y\in \F_2^D$ are adjacent if and only if $w(x-y)=1$.
\end{defn}

Note that $H(D,2)$ is a distance-regular graph of diameter $D$ whose intersection numbers \cite{DRG_book1989,BannaiIto1984,TerAlgebraIII} are
$$
a_i=0,\qquad
b_i=D-i,\qquad
c_i=i \qquad
\hbox{for $i=0,1,\ldots,D$}.
$$

For $i=0,1,\ldots,D$ the $i^{{\rm \,th}}$ distance matrix $\bm A_i$ of $H(D,2)$ is a $0$-$1$ matrix in $\Mat_{\F_2^D}(\C)$ given by
\begin{gather}\label{eq:Ai_cube}
(\bm  A_i)_{xy}=1
  \quad
  \hbox{if and only if}
  \quad
  w(x-y)=i
\end{gather}
for all $x,y\in \F_2^D$. The eigenvalues of $\bm A=\bm A_1$ are
\begin{equation*}
\bm{\theta}_i=D-2i\quad\hbox{for $i=0,1,\ldots,D$}.
\end{equation*}
Let $\bm{E}_i$ denote the primitive idempotent of $H(D,2)$ associated with $\bm{\theta}_i$ for $i=0,1,\ldots,D$.

The distance-regular graph $H(D,2)$ is $Q$-polynomial with respect to the ordering $\{\bm E_i\}_{i=0}^D$ whose dual intersection numbers \cite{BannaiIto1984,TerAlgebraIII} are
$$
a_i^*=0, \qquad
b_i^*=D-i, \qquad
c_i^*=i \qquad
\hbox{for $i=0,1,\ldots,D$}.
$$
Let  $\bm{A}_i^*=\bm{A}^*_i(x)$ denote the $i^{\rm \, th}$ dual distance matrix of $H(D,2)$ with respect to $x$ for $i=0,1,\ldots, D$.
The diagonal matrix $\bm A^*=\bm A^*_1$ is given by
\begin{gather}\label{eq:dual ad of hyper}
\bm A^*_{yy}=D-2w(x-y)
\qquad
\hbox{for all $y\in \F_2^D$}.
\end{gather}

Let $t$ denote an indeterminate over $\C$. In view of the intersection numbers and the dual intersection numbers of $H(D,2)$, we consider the polynomials $\{v_i(t)\}_{i=0}^D$ given by the following recurrence relation:
\begin{align*}
tv_i(t)= (D-i+1) v_{i-1}(t)+(i+1) v_{i+1}(t)
\qquad \hbox{for all $i=1,2,\ldots,D-1$}
\end{align*}
with $v_0(t)=1$ and $v_1(t)=t$.
Recall that for any nonzero $q\in \C$ and any integer $n\geq 1$, the {\it Krawtchouk polynomials}  are
$$
K_i(t;q,n)
=\sum_{j=0}^i
(-q)^j
\frac{{i\choose j}
{t\choose j}}{{n\choose j}}
\qquad
\hbox{for all $i=0,1,\ldots,n$}.
$$
Applying the three-term recurrence of the Krawtchouk polynomials \cite[\S 9.11]{Koe2010} it is routine to verify that
\begin{gather}\label{e:vi}
v_i(t)={D\choose i}
\cdot
K_i
\left(
\frac{D-t}{2};2,D
\right)
\qquad
\hbox{for all $i=0,1,\ldots,D$}.
\end{gather}
The following lemma is immediate from the construction of $\{v_i(t)\}_{i=0}^D$.

\begin{lem}\label{lem:AiAi*_cube}
\begin{enumerate}
\item $\bm A_i=v_i(\bm A)$ for all $i=0,1,\ldots, D$.

\item  $\bm A_i^*=v_i(\bm A^*)$ for all $i=0,1,\ldots, D$.
\end{enumerate}
\end{lem}

Let $\bm{\T}=\bm{\T}(x)$ denote the Terwilliger algebra of $H(D,2)$ with respect to $x$. We now recall some results on $\bm{\T}$.

\begin{prop}
[Corollary 6.8, \cite{hypercube2002}]
\label{prop:irreducible module of hypercube}
For any integer $k$ with $0\leq k\leq \floor{\frac{D}{2}}$, there exists a $(D-2k+1)$-dimensional irreducible $\bm{\T}$-module $L_k$ that has a basis with respect to which the matrices representing $\bm{A}$ and $\bm{A}^*$ are
$$\begin{pmatrix}
    0 & \gamma_1 &  &  & \bm{0} \\
    \beta_0 & 0 & \gamma_2 &  &  \\
     & \beta_1 & 0 & \ddots &  \\
     &  & \ddots & \ddots &\gamma_{D-2k} \\
    \bm{0} &  &  &\beta_{D-2k-1} &0
  \end{pmatrix},
  \qquad
  \begin{pmatrix}
                       \theta^*_0 &  &  &  & \bm{0} \\
                        & \theta^*_1 &  &  &  \\
                        &  & \theta^*_2 &  &  \\
                        &  &  & \ddots &  \\
                       \bm{0} &  &  &  & \theta^*_{D-2k}
                     \end{pmatrix},
$$
respectively, where
\begin{align*}
  \beta_i &= i+1
  \qquad(0\leq i\leq D-2k-1),\\
  \gamma_i &= D-i-2k+1
  \qquad
  (1\leq i\leq D-2k), \\
  \theta_i^* &= D-2(i+k)
  \qquad
  (0\leq i\leq D-2k).
\end{align*}
\end{prop}

\begin{thm}
[Theorem 10.2, \cite{hypercube2002}]
\label{thm:decomposition of hypercube}
The standard $\bm{\T}$-module $\C^{\F_2^D}$ is isomorphic to
\begin{equation*}
\bigoplus\limits_{k=0}^{\floor{\frac{D}{2}}}\frac{D-2k+1}{D-k+1}\binom{D}{k}\cdot L_{k}.
\end{equation*}
\end{thm}

\begin{thm}
[Corollary 6.9, \cite{hypercube2002}]
The $\bm \T$-modules
$
L_k
$
for all integers $k$ with $0\leq k\leq \floor{\frac{D}{2}}$
are all irreducible $\bm \T$-modules up to isomorphism. Moreover these $\bm \T$-modules are mutually non-isomorphic.
\end{thm}

\begin{thm}
[Theorem 14.4, \cite{hypercube2002}]
The algebra $\bm{\T}$ is isomorphic to
$$
\bigoplus\limits_{k=0}^{\floor{\frac{D}{2}}}\Mat_{D-2k+1}(\C).
$$
In particular $\dim \bm{\T}={D+3\choose 3}$.
\end{thm}

Note that $H(D,2)$ is an example of the $D$-dimensional Hamming graphs. In the case of the $D$-dimensional Hamming graphs, the decomposition formula for the standard modules was recently given in \cite[Theorem 1.7]{Huang:CG&Hamming} and \cite[p. 17]{Hamming:2021}.

\section{The Terwilliger algebra of the halved cube}\label{sec:standard T module of halved cube}

Given a nonempty subset $Y$ of $\F_2^D$, we identify $\C^Y$ as a subspace of $\C^{\F_2^D}$ via the natural injection $\C^Y\to \C^{\F_2^D}$ and we view any $M\in \Mat_Y(\C)$ as the linear map $\C^Y\to \C^Y$ that sends $v$ to $Mv$ for all $v\in \C^Y$. Recall from \S\ref{s:introduction} that $A$ denotes the adjacency matrix of $\frac{1}{2}H(D,2)$.

\begin{lem}\label{lem:A_halved}
$A=\bm A_2|_{\C^X}$.
\end{lem}
\begin{proof}
Immediate from Definition \ref{defn:halved cube} and (\ref{eq:Ai_cube}).
\end{proof}

By Lemma \ref{lem:AiAi*_cube}(i) we have $\bm A_2=v_2(\bm A)$.
Combined with Lemma \ref{lem:A_halved} it follows that the eigenvalues of $A$ are $v_2(\bm \theta_i)=\theta_i$ for $i=0,1,\ldots,\floor{\frac{D}{2}}$.
Recall from \S\ref{s:introduction} that $E_i$ denotes the primitive idempotent of $\frac{1}{2}H(D,2)$ associated with $\theta_i$ for $i=0,1,\ldots,\floor{\frac{D}{2}}$.

\begin{lem}\label{lem:Ei_halved}
\begin{enumerate}
\item Suppose that $D$ is odd. Then
$$
E_i=
(\bm E_i+\bm E_{D-i})|_{\C^X}
\qquad
\hbox{for $i=0,1,\ldots, \frac{D-1}{2}$}.
$$

\item Suppose that $D$ is even. Then
$$
E_i=
\left\{
\begin{array}{ll}
(\bm E_i+\bm E_{D-i})|_{\C^X}
\qquad &\hbox{for $i=0,1,\ldots,\frac{D}{2}-1$},
\\
\bm E_{\frac{D}{2}} |_{\C^X}
\qquad &\hbox{for $i=\frac{D}{2}$}.
\end{array}
\right.
$$
\end{enumerate}
\end{lem}
\begin{proof}
(i):
Let $\bm E_i'=\bm E_i+\bm E_{D-i}$ for $i=0,1,\ldots,\frac{D-1}{2}$. Observe that
$
\sum\limits_{i=0}^{\frac{D-1}{2}}
\bm E_i'=
\sum\limits_{i=0}^D\bm E_i$ is equal to the identity matrix and
$
\bm A_2 \bm E_i'=\theta_i \bm E_i'
$
for $i=0,1,\ldots,\frac{D-1}{2}$.
It follows that $\bm E_i'$ is a polynomial in $\bm A_2$ for $i=0,1,\ldots,\frac{D-1}{2}$. Therefore $\C^X$ is $\bm E_i'$-invariant. Combined with Lemma \ref{lem:A_halved} this yields that $E_i=\bm E_i'|_{\C^X}$  for $i=0,1,\ldots,\frac{D-1}{2}$.

(ii): Similar to the proof of Lemma \ref{lem:Ei_halved}(i).
\end{proof}

\begin{lem}\label{lem:vD-1}
Let $n$ denote an integer with $0\leq n\leq D$. Then
$$
v_{D-1}(D-2n)
=
(-1)^n(D-2n).
$$
\end{lem}
\begin{proof}
Using (\ref{e:vi}) yields that $v_{D-1}(D-2n)$ is equal to
\begin{align}\label{e:vD-1}
\sum_{j=0}^{D-1}(-2)^j (D-j)
{n\choose j}
=
D\sum_{j=0}^{D-1}(-2)^j
{n\choose j}
+
2n \sum_{j=0}^{D-2}(-2)^j
{n-1 \choose j}.
\end{align}
It follows from the binomial formula that
$
\sum\limits_{j=0}^k
(-2)^j
{k\choose j}=(-1)^k$ for any integer $k\geq 0$.
The lemma follows by evaluating (\ref{e:vD-1}) by using the above equation.
\end{proof}

Recall from \S\ref{s:introduction} that $A^*=A^*(x)$ denotes the dual adjacency matrix of $\frac{1}{2}H(D,2)$ with respect to $x$.

\begin{lem}\label{lem:A*_halved}
$A^*=\bm A^*|_{\C^X}$.
\end{lem}
\begin{proof}
Let $y\in X$ be given.
By Lemma \ref{lem:Ei_halved} and since $D\geq 3$ we have
$$
A^*_{yy}=2^{D-1}
(\bm E_1+\bm E_{D-1})_{yy}
=\frac{1}{2}(\bm A^*+\bm A_{D-1}^*)_{yy}.
$$
Recall from (\ref{eq:dual ad of hyper}) that $\bm A^*_{yy}=D-2w(x-y)$.
By Lemma \ref{lem:AiAi*_cube}(ii) we have
$$
(\bm A_{D-1}^*)_{yy}=v_{D-1}(D-2w(x-y)).
$$
Since $x,y\in X$ the integer $w(x-y)$ is even.  Combined with  Lemma \ref{lem:vD-1} it follows that
$(\bm A_{D-1}^*)_{yy}=\bm A^*_{yy}$.
Therefore $A^*_{yy}=\bm A^*_{yy}$ for all $y\in X$. The lemma follows.
\end{proof}

Recall from \S\ref{s:introduction} that $\T=\T(x)$ stands for the Terwilliger algebra  of $\frac{1}{2} H(D,2)$ with respect to $x$.
Define $\bm{\mathcal S}$ to be the subalgebra of $\bm{\mathcal T}$ generated by $\bm A_2$ and $\bm A^*$. In light of Lemmas \ref{lem:A_halved} and \ref{lem:A*_halved} we may identify $\T$ as
the algebra consisting of all elements $M|_{\C^X}$ for all $M\in \bm{\mathcal S}$.

\section{Proofs of Proposition \ref{prop:irreducible Mk} and Theorems \ref{thm:decomposition of halved cube}--\ref{thm:structure of T}}\label{sec:decomposition}

\begin{lem}\label{lem:idea}
Let $L$ denote a $\bm{{\mathcal T}}$-submodule of $\C^{\F_2^D}$.
Then the following hold:
\begin{enumerate}
\item $L\cap \C^X$ is a $\mathcal T$-submodule of $\C^X$.

\item $L=(L\cap \C^X )\oplus(L\cap\C^{\F_2^D\setminus X})$.
\end{enumerate}
\end{lem}
\begin{proof}
(i): Observe that $\C^X$ is an $\bm{\mathcal S}$-module.
Since $\bm{{\mathcal S}}$ is a subalgebra of $\bm{\T}$ it follows that $L$ is an $\bm{{\mathcal S}}$-module. Hence $L\cap \C^X$ is an $\bm{\mathcal S}$-module. The statement (i) follows.

(ii): Since $\bm A^*$ is diagonalizable in $\C^{\F_2^D}$ it follows that $\bm{A}^*|_L$ is diagonalizable.
By (\ref{eq:dual ad of hyper}) the eigenvectors of $\bm{A}^*$ in $\C^{\F_2^D}$ lie in $\C^X$ or $\C^{\F_2^D\setminus X}$. The statement (ii) follows.
\end{proof}

Given an even integer $k$ with $0\leq k\leq\floor{\frac{D}{2}}$, let $[A]=[A]_k$ and $[A^*]=[A^*]_k$ denote the two matrices given in (\ref{eq:irr module even}), respectively. Given an even integer $k$ with $2\leq k\leq\ceil{\frac{D}{2}}$, let $\langle A\rangle=\langle A\rangle_k$ and $\langle A^*\rangle=\langle A^*\rangle_k$ denote the two matrices given in (\ref{eq:irr module odd}), respectively.

Let $k$ be any integer with $0\leq k\leq\floor{\frac{D}{2}}$.
Recall the $\bm{\T}$-module $L_k$ from Proposition \ref{prop:irreducible module of hypercube}. We regard $L_k$ as a $\bm{\T}$-submodule of $\C^{\F_2^D}$. It follows from Lemma \ref{lem:idea}(i) that $L_k\cap \C^X$ is a $\T$-module.

\bigskip

\noindent {\it Proof of Proposition \ref{prop:irreducible Mk}.}
(i):
Let $k$ be an even integer with $0\leq k\leq\floor{\frac{D}{2}}$. Since the subdiagonal and superdiagonal entries of $[A]$ are nonzero and the diagonal entries of $[A^*]$ are mutually distinct, the $\T$-module $M_k$ is irreducible.

Let $\{v_i\}_{i=0}^{D-2k}$ denote the basis for the $\bm{\T}$-module $L_k$ from Proposition \ref{prop:irreducible module of hypercube}. Using (\ref{eq:dual ad of hyper}) yields that
\begin{gather}\label{basis_even}
\{v_{2i}\}_{i=0}^{\floor{\frac{D}{2}}-k}
\end{gather}
form a basis for $L_k\cap \C^X$. Using Proposition \ref{prop:irreducible module of hypercube} along with Lemmas \ref{lem:A_halved} and \ref{lem:A*_halved}, a direct calculation yields that the matrices representing $A$ and $A^*$ with respect to the basis (\ref{basis_even}) for $L_k\cap \C^X$ are exactly $[A]$ and $[A^*]$, respectively. The existence of $M_k$ follows.

(ii): Let $k$ be an even integer with $2\leq k\leq\ceil{\frac{D}{2}}$. Since  the subdiagonal and superdiagonal entries of $\langle A\rangle$ are nonzero and the diagonal entries of $\langle A^*\rangle$ are mutually distinct, the $\T$-module $N_k$ is irreducible.

Let $\{v_i\}_{i=0}^{D-2k+2}$ denote the basis for the $\bm{\T}$-module $L_{k-1}$ from Proposition \ref{prop:irreducible module of hypercube}.
Using (\ref{eq:dual ad of hyper}) yields that
\begin{gather}\label{basis_odd}
\{v_{2i+1}\}_{i=0}^{\ceil{\frac{D}{2}}-k}
\end{gather}
form a basis for $L_{k-1}\cap \C^X$. Using Proposition \ref{prop:irreducible module of hypercube} along with Lemmas \ref{lem:A_halved} and \ref{lem:A*_halved}, a direct calculation yields that the matrices representing $A$ and $A^*$ with respect to the basis (\ref{basis_odd}) for $L_{k-1}\cap \C^X$ are exactly $\langle A\rangle$ and $\langle A^*\rangle$, respectively. The existence of $N_k$ follows. \hfill $\square$

\bigskip

\noindent {\it Proof of Theorem \ref{thm:decomposition of halved cube}.}
Applying Lemma \ref{lem:idea}(ii) to Theorem \ref{thm:decomposition of hypercube} the $\T$-module $\C^X$ is isomorphic to
$$
\bigoplus\limits_{k=0}^{\floor{\frac{D}{2}}}\frac{D-2k+1}{D-k+1}\binom{D}{k}\cdot\left(L_k\cap \C^X\right).
$$
From the proof of Proposition \ref{prop:irreducible Mk} we see that the $\T$-module $L_k\cap \C^X$ is isomorphic to
$$
\left\{
\begin{array}{ll}
M_k \qquad
&\hbox{for all even integers $k$ with $0\leq k\leq \floor{\frac{D}{2}}$},
\\
N_{k+1} \qquad
&\hbox{for all odd integers $k$ with $1\leq k< \ceil{\frac{D}{2}}$}.
\end{array}
\right.
$$
Note that the $\T$-module $L_k\cap \C^X=\{0\}$ if $k=\frac{D}{2}$ is odd. By the above comments the result follows.
\hfill $\square$

\bigskip

\noindent {\it Proof of Theorem \ref{thm:classification of irreducible modules}.}
Since the standard $\T$-module $\C^X$ contains all irreducible $\T$-modules up to isomorphism, the first assertion is immediate from Theorem \ref{thm:decomposition of halved cube}.

The $\T$-modules $M_k$ for all even integers $k$ with $0\leq k \leq\floor{\frac{D}{2}}$ are non-isomorphic since their dimensions are all distinct. Similarly the $\T$-modules $N_k$ for all even integers $k$ with $2\leq k \leq\ceil{\frac{D}{2}}$ are non-isomorphic. Now suppose that there are two even integers $k$ and $k'$
with $0\leq k\leq \floor{\frac{D}{2}}$ and $2\leq k'\leq \ceil{\frac{D}{2}}$
 such that the $\T$-module $M_k$ is isomorphic to $N_{k'}$. Since $\dim M_k=\floor{\frac{D}{2}}-k+1$ and $\dim N_{k'}=\ceil{\frac{D}{2}}-k'+1$ by Proposition \ref{prop:irreducible Mk}, it follows that $D$ is even and $k=k'$. Hence $[A^*]=\langle A^*\rangle$. Since the diagonal entries of $[A^*]$ are mutually distinct, the diagonal entries of $[A]$ are identical to those of $\langle A\rangle$. This leads to $k=\frac{D}{2}+1$, a contradiction. The second assertion follows.
\hfill $\square$

\bigskip

\noindent {\it Proof of Theorem \ref{thm:structure of T}.}
Since $\T$ is a finite-dimensional semisimple algebra, it follows from Theorem \ref{thm:classification of irreducible modules} that $\T$ is isomorphic to
$$
\bigoplus_{\substack{k=0 \\ \hbox{\tiny $k$ is even}}}^{\floor{\frac{D}{2}}}\End(M_k)\oplus
\bigoplus_{\substack{k=2 \\ \hbox{\tiny $k$ is even}}}^{\ceil{\frac{D}{2}}}\End(N_k).
$$
By Proposition \ref{prop:irreducible Mk} the algebra $\End(M_k)$ is isomorphic to $\Mat_{\floor{\frac{D}{2}}-k+1}(\C)$ for all even integers $k$ with $0\leq k\leq \floor{\frac{D}{2}}$ and the algebra $\End(N_k)$ is isomorphic to $\Mat_{\ceil{\frac{D}{2}}-k+1}(\C)$ for all even integers $k$ with $2\leq k\leq \ceil{\frac{D}{2}}$. Hence $\dim \T$ is equal to
\begin{align*}
\sum\limits_{\substack{k=0 \\ \hbox{\tiny $k$ is even}}}^{\floor{\frac{D}{2}}}\left(\floor{\frac{D}{2}}-k+1\right)^2
+
\sum\limits_{\substack{k=2 \\ \hbox{\tiny $k$ is even}}}^{\ceil{\frac{D}{2}}}\left(\ceil{\frac{D}{2}}-k+1\right)^2
={\floor{\frac{D}{2}}+3 \choose 3}
+
{\ceil{\frac{D}{2}}+1 \choose 3}.
\end{align*}
The result follows.
\hfill $\square$

\subsection*{Acknowledgements}

The research is supported by the Ministry of Science and Technology of Taiwan under the projects MOST 110-2115-M-008-008-MY2 and MOST 109-2115-M-009-007-MY2.

\bibliographystyle{amsplain}
\bibliography{MP}

\end{document}